\documentclass[letterpaper, 10pt]{IEEEtran} 

\usepackage{graphicx}
\usepackage{amsfonts}
\usepackage{amsmath}
\usepackage{amsthm}
\usepackage{amssymb}
\usepackage[dvips]{color}
\usepackage{setspace} 
 \usepackage{hyperref}
 \usepackage{algorithm,algorithmic}
\theoremstyle{definition}\newtheorem{theo}{Theorem}[section]
\theoremstyle{definition}\newtheorem{prop}[theo]{Proposition}
\theoremstyle{definition}\newtheorem{lemma}[theo]{Lemma}
\theoremstyle{definition}\newtheorem{pb}{Problem}
\theoremstyle{definition}\newtheorem{coro}[theo]{Corollary}
\theoremstyle{definition}\newtheorem{de}{Definition}[section]
\theoremstyle{definition}\newtheorem{rem}{Remark}[section]

\def\C{\mathbb{C}}
\def\R{\mathbb{R}}

\def\su{\mathfrak{su}}
\def\SU{\textrm{SU}}
\def\O{\mathcal{O}}
\def\H{\mathcal{H}}
\def\ad{\textrm{ad}}
\def\Ad{\textrm{Ad}}
\def\re{\textrm{Re}}
\def\tr{\textrm{tr}}

\def\End{\textrm{End}_{\rho_0}}
\def\cJ{\mathcal{J}}
\def\A{\mathcal{A}}
\def\eps{u}
\date{}
\title{Gradient flow for controlling quantum ensemble}
\author{Ruixing \textsc{Long}
~\and Herschel \textsc{Rabitz}
 \thanks{Ruixing Long and Herschel Rabitz are with the Department of Chemistry, Princeton University, USA, emails: rlong@princeton.edu, hrabitz@princeton.edu}
}

\begin{document}
\maketitle
\begin{abstract}
We propose in this paper a gradient-type dynamical system to solve the problem of maximizing quantum observables for finite dimensional closed quantum ensembles governed by the controlled Liouville-von Neumann equation. The asymptotic behavior is analyzed: we show that under the regularity assumption on the controls the dynamical system almost always converges to a solution of the maximization problem; we also detail the difficulties related to the occurrence of singular controls.% The validity of the method is supported by numerical simulations.
\end{abstract}

%\tableofcontents

\section{Introduction}

Quantum control is concerned with actively manipulating physical and chemical processes on the atomic or molecular scale where quantum mechanics is the rule. The origin of quantum control goes back to the early attempts to use lasers for selectively breaking molecular bonds, and several approaches using quantum interference, adiabatic passage, pump-dump control etc. have been proposed since 1970's. For the historical development of quantum control, the state of the art from both theoretical and experimental points of view, and open research directions, see for instance the recent review paper \cite{Brif:2010ys} . An overview on control techniques applied to manipulating quantum systems is also given in \cite{Dong:2010fk}. More detailed treatment from a control theoretical point of view can be found in \cite{DAlessandro:2008qf}. Among existing methods for controlling quantum systems, optimal control theory plays a major role. The key point is to develop control strategies in a constructive way such that a certain performance index, or cost functional is optimized under the constraints imposed by realistic experimental conditions. Three classes of problems - state transition, observable maximization, and unitary transformation- have been attracting the most attention in the community \cite{Brif:2010ys, More:2011qf}. The performance indices in these problems only depend on the final states of the corresponding quantum systems, although in full generality time or energy consumption could be taken into account as well (see for instance \cite{Khaneja:2001kx, Lapert:2010vn, Boscain:2004uq}). Moreover, these performance indices can also be used as Lyapunov functions in closed-loop feedback designs for stabilization or trajectory tracking, see \cite{Zhu:2003bh, Mirrahimi:2005ly, Mirrahimi:2005zr, Beauchard:2007ve,Wang:2010ys} and references therein. In this paper, we only consider the problem of maximizing quantum observables for closed quantum systems, an open loop strategy will be proposed. The analysis also extends to state-transition and unitary transformation problems.

For a closed $n-$level quantum system, the evolution of its density matrix $\rho(t)$ under the dipole moment approximation is described by the following time-varying Liouville-von Neumann equation: 
%Given a positive integer $n$, define $\su(n)$ as 
%$$\su(n):=\{\Omega\in\C^{n\times n}, \Omega^\dagger=-\Omega, \textrm{ and } \textrm{tr}(\Omega)=0\}.$$
%Given $(H_0, H_1)\in\su(n)\times\su(n)$, we consider the following control system
\begin{equation}\label{CS}
\left\{\begin{array}{lll}
\dot{\rho}(t)&=&[H_0+\eps(t) H_1, \rho(t)],\qquad t\in [0,T],\\
\rho(0)&=&\rho_0,
\end{array}\right.
\end{equation}where $\rho_0$ is the initial density matrix which is assumed to be Hermitian, the traceless skew-Hermitian matrices $H_0$ and $H_1$ are respectively the free Hamiltonian of the system and the dipole moment. The vector space of traceless skew-Hermitian matrices will be denoted by $\su(n)$. We assume that the admissible controls $\eps$ are elements of $\mathcal{H}:=L^2([0,T],\R)$. This corresponds to the ideal case where the intensity of the external field $\eps(\cdot)$ is not constrained. The vector space $\mathcal{H}$ equipped with the standard inner product $\displaystyle(u,v)_\H=\int_0^Tu(t)v(t)dt$, for $(u,v)\in \H\times\H$, is a Hilbert space. The corresponding norm will be denoted by $\Vert \cdot\Vert_\H$.

It is well-known that the solution of (\ref{CS}) is given by 
\begin{equation}\label{rho-sol}
\rho(t)=U(t)\rho_0U(t)^\dagger,
\end{equation}
where the propagator $U(\cdot)$ satisfies
\begin{equation}\label{CSU}
\left\{\begin{array}{lll}
\dot{U}(t)&=&(H_0+\eps(t) H_1)U(t),\qquad t\in [0,T],\\
U(0)&=&\textrm{Id}.
\end{array}\right.
\end{equation}Since $H_0$ and $H_1$ belong to $\su(n)$, $U(\cdot)$ is a curve in the special unitary group $\SU(n)$. Recall that $\SU(n)$ is a compact Lie group and its Lie algebra is $\su(n)$. Eq. (\ref{rho-sol}) implies that $\rho$ evolves in a subset of the unitary orbit of $\rho_0$ defined by
$$\O(\rho_0):=\{U\rho_0U^\dagger, ~U\in\SU(n)\}.$$ We assume from now that the system (\ref{CSU}) is controllable, then the state space of $\rho$ is equal to $\O(\rho_0)$ and the system \eqref{CS} is controllable in the sense that all points of $\O(\rho_0)$ can be reached from $\rho_0$ by choosing suitable controls. 

\begin{rem}
For $T$ large enough, a necessary and sufficient condition for \eqref{CSU} to be controllable is that the Lie algebra generated by $H_0$ and $H_1$ is equal to $\su(n)$. This is a consequence of the controllability results on general Lie groups obtained by Jurdjevic and Sussmann in \cite{Jurd-Suss}. See \cite{Ramakrishna:1995zr, Fu:2001fk, Schirmer:2001uq}, or  \cite[Ch. 3]{DAlessandro:2008qf} for controllability of quantum systems. We also note that the set of pairs $(H_0, H_1)$ such that $H_0$ and $H_1$ generate $\su(n)$ is open and dense in $\su(n)\times\su(n)$ (cf. \cite[Th. 12, Ch. 6, p 188]{jurdjevic-GCT}).
\end{rem}

Define the end-point map for (\ref{CS}) as
\begin{equation*}
\End(\cdot):\begin{array}{ccc} 
\mathcal{H}&\mapsto& \O(\rho_0)\\
\eps&\rightarrow&\rho(T)
\end{array}.
\end{equation*}

In this paper, we are interested in the following maximization problem:
\begin{pb}\label{main-pb}
Let $\theta$ be a Hermitian matrix. Find $\eps_{\max}\in\H$ maximizing the cost function
\begin{equation}
\mathcal{J}(\eps):=\re~\textrm{tr}(\End(\eps)\theta), \quad \textrm{ for }\eps\in\H.
\end{equation}
\end{pb}

\begin{rem}
$\theta$ represents an observable for the quantum system and $\tr (\rho(T)\theta)$ is the average of different possible results given by the measurement of $\theta$ at time $T$ (cf. \cite[Chap 3-E]{tannoudji}). {\bf Problem 1} consists in finding a control field $u$ maximizing this average.
\end{rem}

This problem is closely related to the two following problems.
%We first consider the two following problems.
\begin{pb}\label{op}
Let $\theta$ be a Hermitian matrix. Find $\rho_{\max}\in\O(\rho_0)$ maximizing the cost function 
\begin{equation}
J(\rho):=\re~\textrm{tr}(\rho\theta),\quad \textrm{ for } \rho\in\O(\rho_0).
\end{equation}
\end{pb}

\begin{pb}\label{mp}
Given an arbitrary target state $\rho^{\textrm{final}}\in\O(\rho_0)$, find $\eps^{\textrm{final}}\in\mathcal{H}$ such that 
\begin{equation}
\End(\eps^{\textrm{final}})=\rho^{\textrm{final}}.
\end{equation}
\end{pb}

\begin{rem}
The compactness of $\O(\rho_0)$ guarantees the existence of solutions for {\bf Problem \ref{op}}, which in turn implies, together with the controllability assumption, the existence of solutions for {\bf Problem \ref{main-pb}}.
\end{rem}

\begin{rem}\label{main-rem}
If we are able to find $\rho_{\max}$ a solution to {\bf Problem \ref{op}}, then {\bf Problem \ref{main-pb}} is equivalent to {\bf Problem \ref{mp}} with target state equal to $\rho_{\max}$. We also note that $\eps$ is a solution to {\bf Problem \ref{main-pb}} if and only if $\End(\eps)$ is a solution to {\bf Problem \ref{op}}.
\end{rem}

We discuss in this paper a gradient-type dynamical system to solve {\bf Problem 1}. The method is well-known in the quantum chemistry and NMR (Nuclear Magnetic Resonance) communities (see for example \cite{Judson:1992vn,Brif:2010ys,khaneja-NMR}). We give here a rigorous mathematical formulation of this method as well as analysis on its asymptotic behavior. We also formulate some open questions related to the presence of singular controls for (\ref{CS}). The paper is organized as follows. We recall in Section \ref{recall} classical results on the geometry of the unitary orbit $\O(\rho_0)$ and derive some computational lemmas related to the end-point map.  The main results of this paper concerning the asymptotic behavior of  the dynamical system are presented in Section \ref{sec:gradient-flow}. %, where we also present some numerical simulations. 
Finally, concluding remarks are formulated in Section \ref{concl} and Appendix deals with a technical proof. 

%=================================== Optimization on orbit ==================================================

\section{Preliminary results}\label{recall}

\subsection{Geometry of the unitary orbit}

We summarize in this paragraph some results on the geometry of the unitary orbit $\O(\rho_0)$. The key point is to define a suitable Riemannian metric on $\O(\rho_0)$. Add references. The presentation here follows \cite[Section 3.4.4]{Schulte-Herbruggen:fk}.

Recall that $\O(\rho_0)$ is a compact connected submanifold of $\C^{n\times n}$ isomorphic to the quotient space $\SU(n)/\bf{H}$, where 
$${\bf H}:=\{U\in\SU(n),~U\rho_0U^\dagger=\rho_0\}$$ denotes the stabilizer group of $\rho_0$. We have
\begin{equation}\label{dim}
\textrm{dim }\O(\rho_0)=n^2-1-\textrm{dim }{\bf H}:=N.
\end{equation}

The tangent space of $\O(\rho_0)$ at $\rho=\Ad_U\rho_0:=U\rho_0U^\dagger$ is given by 
\begin{equation}\label{tangent-space}
T_\rho\O(\rho_0)=\ad_\rho\su(n)=\{\ad_\rho\Omega, ~\Omega\in\su(n)\},
\end{equation}with $\ad_\rho\Omega:=[\rho,\Omega]:=\rho\Omega-\Omega\rho$. 

\begin{rem}
Since the adjoint map $\Ad_U:\Omega\mapsto\Ad_U\Omega$ defines an automorphism on $\su(n)$, the tangent space $T_\rho\O(\rho_0)$ is also equal to $\{\ad_\rho\Ad_U\Omega,~\Omega\in\su(n)\}$.
\end{rem}

In order to define the gradient of the cost function $J$, we first need to equip $T_\rho\O(\rho_0)$ with a scalar product. Note that the kernel of $\ad_\rho:\su(n)\mapsto\C^{n\times n}$ is given by
$$\mathfrak{h}:=\{\Omega\in\su(n),~[\rho_0,\Omega]=0\}$$ and forms the Lie subalgebra to ${\bf H}$.  By the standard Hilbert-Schmidt scalar product $(\Omega_1,\Omega_2)\mapsto \tr(\Omega_1^\dagger\Omega_2)$ on $\su(n)$ one can define the ortho-complement of $\mathfrak{h}$ as 
$$\mathfrak{p}:=\{\Omega_1\in\su(n), ~\tr(\Omega_1^\dagger\Omega_2)=0,\textrm{ for all }\Omega_2\in\mathfrak{h}\}.$$This induces a unique decomposition of any skew-Hermitian matrix $\Omega=\Omega^{\mathfrak{h}}+\Omega^{\mathfrak{p}}$ with $\Omega^{\mathfrak{h}}\in\mathfrak{h}$ and $\Omega^{\mathfrak{p}}\in\mathfrak{p}$. 

\begin{de}\label{Riem-metric}
For $\rho=\Ad_U\rho_0$ with $U\in\SU(n)$, we define a scalar product $\langle\cdot,\cdot\rangle_{\rho}$ on $T_\rho\O(\rho_0)$ by
\begin{equation}\label{metric1}
\langle\ad_\rho(\Ad_U\Omega_1),\ad_\rho(\Ad_U\Omega_2)\rangle_\rho:=\tr(\Omega_1^{\mathfrak{p}\dagger}\Omega_2^{\mathfrak{p}}),
\end{equation} which is equivalent to 
\begin{equation}\label{metric2}
\langle\ad_\rho\Omega_1,\ad_\rho\Omega_2\rangle_\rho:=\tr(\Omega_1^{\mathfrak{p}_\rho\dagger}\Omega_2^{\mathfrak{p}_\rho})
\end{equation}
 with $\mathfrak{p}_\rho:=\Ad_U\mathfrak{p}$.
\end{de}

A fundamental property of the Riemannian metric defined above is it is $\Ad_\SU(n)-$invariant, i.e., $\forall~ \xi,\eta\in T_\rho\O(\rho_0),$ and $\forall~ U\in\SU(n),$
\begin{equation}\label{Ad-invariance}
\langle\xi,\eta\rangle_\rho=\langle\Ad_U\xi,\Ad_U\eta\rangle_{\Ad_U\rho}.
\end{equation}
For later use, we recall the following result.
\begin{prop}[Theorem 3.16 \cite{Schulte-Herbruggen:fk}]\label{grad-formula}
Let $J$ be the cost function considered in {\bf Problem \ref{op}} and $\rho\in\O(\rho_0)$. Then, the gradient of $J$ at $\rho$ with respect to the Riemannian metric defined by Eq. (\ref{metric1}) is given by
$$\nabla J(\rho)=[\rho,[\rho,\theta]].$$

Furthermore, $\rho_c\in\O(\rho_0)$ is a critical point of $J$ if and only if $$[\rho_c,\theta]=0.$$

\end{prop}\smallskip

\begin{rem}
Let $dJ(\rho)$ be the differential of $J$ at $\rho$. Then, by definition, we have 
$$\forall~ \eta\in T_\rho\O(\rho_0),~dJ(\rho)\eta=\langle\nabla J(\rho),\eta\rangle_\rho.$$We note that the expression of the gradient depends on the metric chosen for $T_\rho\O(\rho_0)$. One can choose metrics other than the one defined by (\ref{metric1}), for example, the induced Riemannian metric if we consider $\O(\rho_0)$ as a submanifold embedded in $\C^{n\times n}$. However, the invariant metric defined above gives a simple expression of $\nabla J$.
%and simplifies the analysis of the corresponding gradient flow (\ref{flow}). See also Propositions \ref{grad-flow} and \ref{convergence}.
\end{rem}\smallskip

In order to simplify the discussion, we assume that
\begin{itemize}
\item[({\bf H1})] the initial density matrix $\rho_0$ and the observable $O$ both have simple eigenvalues.
\end{itemize}

The following result is a direct consequence of Proposition \ref{grad-formula} and {\bf (H1)}.
\begin{coro}\label{critical-j}
Under ({\bf H1}), $J$ has $M:=n!$ isolated critical points in $\mathcal{O}(\rho_0)$.  
\end{coro}
Let $\{\rho_i\}_{i=1,\dots,M}$ be the critical points of $J$ such that $J(\rho_1)\leq\cdots\leq J(\rho_M)$. For $i=1,\dots, M$, let $\nabla^2J(\rho_i)$ be the Hessian of $J$ at $\rho_i$.
\begin{lemma}\label{hessian-j}
Under ({\bf H1}), for $i=1,\dots, M$, $\nabla^2J(\rho_i)$ is non-degenerate. Moreover, $\nabla^2J(\rho_1)$ is positive definite, $\nabla^2J(\rho_M)$ is negative definite, and $\nabla^2J(\rho_i)$ is not definite for $i=2,\dots, M-1$.   
\end{lemma}
Lemma \ref{hessian-j} states that $J$ only has one minimum and one maximum, all other critical points are saddles. The proof is a straightforward adaptation of the one for \cite[Th. 1.3, p 52]{helmke-moore}. See also \cite[Cor. 3.8]{Schulte-Herbruggen:fk} and its proof.

\subsection{Differential of the end-point map and its adjoint operator}
The end-point map $\End(\cdot)$ is $C^{\infty}$ (in fact analytical in our case). For $u\in\H$, the first derivative of $\End$ at $u$ is given by
\begin{equation}
d\End(u):\begin{array}{lll}
\H&\mapsto&T_{\End(u)}\O(\rho_0)\smallskip\\
v&\mapsto&d\End(u)v=y_v(T)
\end{array},
\end{equation}where, for every $v\in\H$, $y_v:[0,T]\mapsto T\O(\rho_0)$ is the solution of the variational equation
\begin{equation}\label{var-eq}
\left\{\begin{array}{lll}
\dot{y}(t)&=&[H_0+u(t)H_1,y(t)]+v(t)[H_1,\rho(t)],~Êt\in[0,T],\\
y(0)&=&0,
\end{array}\right.
\end{equation}with $\rho(\cdot)$ denoting the solution of (\ref{CS}) associated with the control $u$. The following computational lemma is obtained by variation of constants.

\begin{lemma}\label{sol:var-eq}
If $U(\cdot):[0,T]\mapsto\SU(n)$ satisfies
\begin{equation*}
\left\{\begin{array}{lll}
\dot{U}(t)&=&(H_0+u(t) H_1)U(t),\qquad t\in [0,T],\\
U(0)&=&\textrm{Id},
\end{array}\right.
\end{equation*}then, $y_v(\cdot):[0,T]\mapsto T\O(\rho_0)$ given by
\begin{equation}\label{sol}
y_v(t)=U(t)\int_0^t[U^\dagger(s) H_1U(s), \rho_0]v(s)ds~U(t)^\dagger
\end{equation}
is the solution of (\ref{var-eq}).
\end{lemma}

%\begin{proof}[Proof of Lemma \ref{sol:var-eq}]
%We set $\displaystyle z(t):=\int_0^t[U^\dagger(s) H_1U(s), \rho_0]v(s)ds$. Then, we have 
%$$z(0)=0\quad\textrm{ and }\quad\dot{z}(t)=[U^\dagger(t) H_1U(t), \rho_0]v(t). $$
%Let $y_v(t):=U(t)z(t)U(t)^\dagger$. We have
%\begin{eqnarray*}
%\dot{{y}}(t)&=&\dot{U}(t)z(t)U(t)^\dagger+U(t)z(t)\dot{U}(t)^\dagger+U(t)\dot{z}(t)U(t)^\dagger\\
%&=&(H_0+u(t) H_1)U(t)z(t)U(t)^\dagger\\
%&&+U(t)z(t)[(H_0+u(t) H_1)U(t)]^\dagger\\
%&&+v(t)U(t)[U^\dagger(t) H_1U(t), \rho_0]U(t)^\dagger\\
%&=&[H_0+u(t)H_1,y(t)]+v(t)[H_1,\rho(t)].
%\end{eqnarray*}
%Since $y_v(0)=z(0)=0$, $y_v(t)=U(t)z(t)U(t)^\dagger$ is the unique solution of Eq. (\ref{var-eq}).
%
%\end{proof}\smallskip

\begin{coro}\label{estimate-var-eq-order1}
There exists a contant $\tilde{C}>0$ depending on $\rho_0$, $H_1$, $T$ such that for all $u\in\H$, we have
\begin{equation}
\Vert d\End(u)v\Vert\leq \tilde{C}\Vert v\Vert_\H,\quad \forall~Êv\in\H.
\end{equation}
\end{coro}
\begin{proof}[Proof of Corollary \ref{estimate-var-eq-order1}] 
It suffices to note that there exists a constant $C_1>0$ such that 
$$\Vert z(t)\Vert\leq C_1\Vert v\Vert_\H,\quad\textrm{ for all }t\in[0,T],\textrm{ and }u\in\H,$$where $\displaystyle z(t)=\int_0^t[U^\dagger(s) H_1U(s), \rho_0]v(s)ds$.
\end{proof}
\medskip

%Finally, we recall the notion of regular and singular controls. They will play a crucial role in the convergence analysis of the gradient flow. 
\begin{de}\label{de:regular-control}
A control $u\in\H$ is called {\it regular } if the rank of $d\End(u)$ is equal to the dimension of the state space $\O(\rho_0)$.  The corresponding trajectory is called {\it regular trajectory}. 
\end{de}

\begin{de}\label{de:singular-control}
A control $u\in\H$ is called {\it singular} if the rank of $d\End(u)$ is smaller than the dimension of the state space $\O(\rho_0)$.  The corresponding trajectory is called {\it singular trajectory}. The {\it co-rank} of a singular control $u$ is defined as equal to $$\textrm{dim }\O(\rho_0)-\textrm{ rank }(d\End(u)).$$
\end{de}

\begin{rem}
The notion of regular and singular controls will play a crucial role in the convergence analysis of the gradient flow, see Section \ref{sec:gradient-flow} for more detail.
\end{rem}

\begin{de}
Given $u\in\H$, let $\rho(\cdot)$ be the solution of
\begin{equation}
\left\{\begin{array}{lll}
\dot{\rho}(t)&=&[H_0+u(t) H_1, \rho(t)],\qquad t\in [0,T],\\
\rho(0)&=&\rho_0.
\end{array}\right.
\end{equation}The {\it adjoint equation} along $\rho(\cdot)$ is defined by
\begin{equation}\label{adj}
\left\{\begin{array}{lll}
\dot{q}(t)&=&[H_0+u(t) H_1, q(t)],\qquad t\in [0,T],\\
q(T)&=&q_T,
\end{array}\right.
\end{equation} for some $q_T\in T_{\rho(T)}\O(\rho_0)$. The solution $q(\cdot)$ of Eq. (\ref{adj}) is called {\it adjoint vector}. The corresponding {\it switching function} $\Phi_{\rho_0,q_T}(\cdot)$ is defined by
\begin{equation}
\Phi_{\rho_0,q_T}(t):=\langle q(t), [H_1, \rho(t)]\rangle_{\rho(t)},
\end{equation}where the Riemannian metric $\langle~,~\rangle_{\rho(t)}$ is chosen to be the one given in Definition \ref{Riem-metric}.
\end{de}\smallskip

\begin{lemma}\label{switch}
For $q_T\in  T_{\rho(T)}\O(\rho_0)$ and $v\in\H$, we have
\begin{equation}
\langle q_T,d\End(u)v\rangle_{\rho(T)}=(\Phi_{\rho_0,q_T}, v)_\H.
\end{equation}
\end{lemma}

\begin{proof}[Proof of Lemma \ref{switch}]
\begin{eqnarray*}
&&\langle q_T,d\End(u)v\rangle_{\rho(T)}\\
&=&\langle q_T,U(T)~\int_0^T[U(s)^\dagger H_1U(s),\rho_0]v(s)ds~U(T)^\dagger\rangle_{\rho(T)}\\
&=&\int_0^T\langle q_T, U(T)[U^\dagger(s)H_1U(s),\rho_0]U^\dagger(T)\rangle_{\rho(T)}v(s)ds.
\end{eqnarray*}

Since the Riemannian metric $\langle~,~\rangle_{\rho(t)}$ is $\Ad_{SU}-$invariant, we have
\begin{eqnarray}
&&\Phi_{\rho_0,q_T}(t)=\langle q(t),[H_1,\rho(t)]\rangle_{\rho(t)}\nonumber\\
&=&\langle~\Ad_{U(T-t)}q(t),~\Ad_{U(T-t)}[H_1,\rho(t)]~\rangle_{\Ad_{U(T-t)}\rho(t)}\nonumber\\
&=&\langle q_T,~U(T-t)[H_1,\rho(t)]U(T-t)^\dagger\rangle_{\rho(T)}\nonumber\\\
%&=&\langle q_T,~U(T)U(t)^\dagger\big(H_1U(t)\rho_0U(t)^\dagger-U(t)\rho_0U(t)^\dagger H_1\big)U(t)U(T)^\dagger\rangle_{\rho(T)}\nonumber\\
&=&\langle q_T,~U(T)[U(t)^\dagger H_1U(t),\rho_0]U(T)^\dagger\rangle_{\rho(T)}.\label{expression-phi}
\end{eqnarray}This implies
%\begin{equation*}
$\langle q_T,d\End(u)v\rangle_{\rho(T)}=(\Phi_{\rho_0,q_T}, v)_\H.$
%\end{equation*}
\end{proof}\smallskip

\begin{de}\label{transp}
The adjoint operator $d\End^{\bf *}(u)$ of $d\End(u)$ is defined as the unique operator satisfying
\begin{equation*}
\langle z, d\End(u)v\rangle_{\End(u)}=(d\End^{\bf *}(u)z,v)_\H,
\end{equation*}for all $z\in T_{\End(u)}\O(\rho_0)$ and $v\in\H$.
\end{de}\smallskip

%\begin{rem}
%We note that $d\End^*(u)$ maps from $T_{\End(u)}\O(\rho_0)$ to $\H$.
%\end{rem}

From Lemma \ref{switch} and Definition \ref{transp}, we immediately get the following corollary.

\begin{coro}\label{equal}
For $z\in T_{\End(u)}\O(\rho_0)$, we have 
\begin{equation}
d\End^*(u)z=\Phi_{\rho_0,z}.
\end{equation}
\end{coro}\bigskip

\begin{de}\label{gramian}
For $u\in\H$, the non-negative symmetric matrix called {\it controllability Gramian} of (\ref{var-eq}) is defined by 
$$G(u):=d\End(u)d\End^*(u).$$
\end{de}The following fundamental property holds.
\begin{prop}\label{gram}
For all $z\in T_{\End(u)}\O(\rho_0)$, we have
$$\langle z,G(u)z\rangle_{\End(u)}=\Vert d\End^*(u)z\Vert_\H^2=\Vert \Phi_{\rho_0,z}\Vert^2_\H,$$ and $$\textrm{rank }d\End(u)=\textrm{dim }\O(\rho_0)\Longleftrightarrow G(u) \textrm{ is positive definite.}$$
\end{prop}\bigskip

For later use, we finish this section by giving the second derivative of $\End(\cdot)$ at $u$ in the direction $v\in\H$.
\begin{equation}
d^2\End(u):\begin{array}{lll}
\H&\mapsto&T_{\End(u)}\O(\rho_0)\smallskip\\
v&\mapsto&d^2\End(u)(v,v)=r_v(T)
\end{array},
\end{equation}where, for every $v\in\H$, $r_v:[0,T]\mapsto T\O(\rho_0)$ is the solution of the second variational equation
\begin{equation}\label{var-eq2}
\left\{\begin{array}{lll}
\dot{r}(t)&=&[H_0+u(t)H_1,r(t)]+v(t)[H_1,y_v(t)],~t\in[0,T],\\
r(0)&=&0,
\end{array}\right.
\end{equation}with $y_v(\cdot)$ denoting the solution of the first variational equation (\ref{var-eq}).\medskip

The following lemma is straightforward.
\begin{lemma}\label{sol:var-eq2}
If $U(\cdot):[0,T]\mapsto\SU(n)$ satisfies
\begin{equation*}
\left\{\begin{array}{lll}
\dot{U}(t)&=&(H_0+u(t) H_1)U(t),\qquad t\in [0,T],\\
U(0)&=&\textrm{Id},
\end{array}\right.
\end{equation*}then, $r_v(\cdot):[0,T]\mapsto T\O(\rho_0)$ given by
\begin{equation}\label{sol}
r_v(t)=U(t)\int_0^t[U^\dagger(s) H_1U(s), z_v(t)]v(s)ds~U(t)^\dagger,
\end{equation}with $\displaystyle z_v(\cdot):=\int_0^t[U^\dagger(s)H_1U(s),\rho_0])v(s)ds$
is the solution of (\ref{var-eq2}).
\end{lemma}

\begin{coro}\label{estimate-var-eq2}
There exists a contant $\tilde{C}>0$ depending on $\rho_0$, $H_1$, $T$ such that for all $u\in\H$, we have
\begin{equation}
\Vert d^2\End(u)(v,v)\Vert\leq \tilde{C}\Vert v\Vert_\H^2,\quad \forall~Êv\in\H.
\end{equation}
\end{coro}

%\begin{proof}[Proof of Corollary \ref{estimate-var-eq2}]
%The proof is straightforward from Corollary \ref{estimate-var-eq-order1}].
%
%\end{proof}

\section{Gradient flow in $\H$}\label{sec:gradient-flow}

A natural idea to tackle {\bf Problem 1}, which is an optimization problem in the infinite dimensional control space $\H$, is to follow the gradient of $\mathcal{J}$ as an ascent direction in order to increase $\mathcal{J}$. 
The purpose of this section is to present in a rigorous way a gradient-type algorithm widely used in the quantum chemistry and NMR communities, see for example \cite{Judson:1992vn,Brif:2010ys,khaneja-NMR}.

\subsection{Description of the method and some general properties}
We first compute the gradient of $\mathcal{J}$. Note that $\mathcal{J}(u)=J(\End(u))$. 
\begin{lemma}\label{gradJ}
For $u\in\H$, we have
\begin{equation}\label{eq:gradJ}
\nabla\mathcal{J}(u)=d\End^*(u)\nabla J(\End(u)).
\end{equation}
\end{lemma}\smallskip

\begin{proof}[Proof of Lemma \ref{gradJ}]
Given $u\in\H$, for any $v\in\H$, we have
\begin{eqnarray*}
d\mathcal{J}(u)v&=&dJ(\End(u))d\End(u)v\\
&=&\langle \nabla J(\End(u)), d\End(u)v\rangle_{\End(u)}\\
&=&(d\End^*(u)\nabla J(\End(u)), v)_\H.
\end{eqnarray*}By defintion, we have
$$\nabla\mathcal{J}(u)=d\End^*(u)\nabla J(\End(u)).
$$
\end{proof}\smallskip

\begin{algorithm}[H]\label{algo:grad-flow}
\caption{Gradient Flow}
\begin{itemize}
\item[(i)] Choose an arbitrary control $u_0\in \H$.
\item[(ii)] Solve the following initial value problem 
\begin{equation}\label{IVP-grad-flow1}
\left\{\begin{array}{lll}
\displaystyle \frac{d\Pi}{ds}(s)&=&\nabla\mathcal{J}(\Pi(s))\smallskip\\
 \Pi(0)&=&u_0
\end{array}\right.,
\end{equation}or more precisely,
\begin{equation}\label{IVP-grad-flow}
\left\{\begin{array}{lll}
\displaystyle \frac{d\Pi}{ds}(s)&=&d\End^*(\Pi(s))\displaystyle\nabla J(\End(\Pi(s)))\smallskip\\
 \Pi(0)&=&u_0
\end{array}\right..
\end{equation}
\end{itemize}
\end{algorithm}

Before giving some preliminary analysis on the algorithm, we first explain how to compute the right-hand side of Eq. (\ref{IVP-grad-flow}).
\begin{lemma}\label{eq:right-hand-grad}
For $u\in\H$, we have
\begin{eqnarray*}
&&d\End^*(u)\displaystyle\nabla J(\End(u))\\
&=&-\tr~([\rho_0,U(T)^\dagger\theta U(T)]U^\dagger(t)H_1U(t)),
\end{eqnarray*}where $U(\cdot)$ satisfies
\begin{equation*}
\left\{\begin{array}{lll}
\dot{U}(t)&=&(H_0+u(t) H_1)U(t),\qquad t\in [0,1],\\
U(0)&=&\textrm{Id}.
\end{array}\right.
\end{equation*}

\end{lemma}\smallskip

\begin{proof}[Proof of Lemma \ref{eq:right-hand-grad}]
By Corollary \ref{equal}, it is equivalent to compute $\Phi_{\rho_0,\nabla J(\End(u))}$. Eq. (\ref{expression-phi}) implies that
\begin{eqnarray*}
&&\Phi_{\rho_0,\nabla J(\End(u))}\\
&=&\langle \nabla J(\rho(T)),~U(T)[U(t)^\dagger H_1U(t),\rho_0]U(T)^\dagger\rangle_{\rho(T)}\\
%&=&\langle \Ad_{U^\dagger(T)}\nabla J(\rho(T)),~\Ad_{U^\dagger(T)}U(T)[U(t)^\dagger H_1U(t),\rho_0]U(T)^\dagger\rangle_{\Ad_{U^\dagger(T)}\rho(T)}\\
&=&\langle [\rho_0,U^\dagger(T)[\rho(T),\theta]U(T)],~Ê[U(t)^\dagger H_1U(t),\rho_0]\rangle_{\rho_0}\\
&=&-\langle\ad_{\rho_0}[\rho_0,U^\dagger(T)\theta U(T)],~Ê\ad_{\rho_0}U(t)^\dagger H_1U(t)\rangle_{\rho_0}.
\end{eqnarray*}

We note that $[\rho_0, U^\dagger (T)\theta U(T)]\in\mathfrak{p}$. In fact, we have $[\rho_0,\gamma]\in\mathfrak{p}$ for all $\gamma\in\C^{n\times n}$. Indeed, let $\omega\in\mathfrak{h}$. By definition of $\mathfrak{h}$, we have $[\rho_0,\omega]=0$. Since $\tr([\rho_0,\gamma]\omega)=-\tr(\gamma[\rho_0,\omega])$, we get $[\rho_0,\gamma]\in\mathfrak{p}$. Therefore, by the definition of $\langle\cdot,\cdot\rangle_{\rho_0}$, we have $\Phi_{\rho_0,\nabla J(\End(u))}=-\tr~([\rho_0,U(T)^\dagger\theta U(T)]U^\dagger(t)H_1U(t)).$
%\end{equation*}
\end{proof}\smallskip

\begin{prop}\label{prop:grad-flow}
The initial value problem defined by Eq. (\ref{IVP-grad-flow}) has a unique solution which is globally defined for all $s\geq 0$.
\end{prop}

\begin{proof}[Proof of Proposition \ref{prop:grad-flow}]
The uniqueness and local existence of solution for Eq. (\ref{IVP-grad-flow}) is straightforward. If $u$ is not a critical point of $\mathcal{J}$, since $U(\cdot)\in\SU(n)$, Lemma \ref{eq:right-hand-grad} implies that there exists a constant $C>0$ depending on $\rho_0$, $\theta$, $H_1$, and the final time $T$ such that
\begin{equation}\label{eq:estimate}
\Vert \nabla\mathcal{J}(u)\Vert_\H\leq C,\quad\forall~u\in\H.
\end{equation}Then, by Cauchy-Schwartz inequality, we have
\begin{eqnarray*}
\frac{d\Vert \Pi(s)\Vert_\H}{ds}&=&(\frac{\Pi(s)}{\Vert \Pi(s)\Vert_\H},\frac{d\Pi(s)}{ds})_\H\leq C.
\end{eqnarray*}Finally, Growall inequality implies that
\begin{equation}\label{eq:estimate-sol}
\Vert \Pi(s)\Vert_\H\leq Cs.
\end{equation}
Therefore, the solution of Eq. (\ref{IVP-grad-flow}) is globally defined on $[0,\infty[$. 
\end{proof}

%\begin{lemma} \label{lemma:increase}
%If $\Pi(\cdot)$ is the solution of Eq. (\ref{IVP-grad-flow}) and $u_0$ is not a critical point of $\mathcal{J}$, then $\mathcal{J}(\Pi(\cdot))$ is a strictly increasing function of $s$.
%\end{lemma}
%
%\begin{proof}[Proof of Lemma \ref{lemma:increase}]
%It suffices to note that if $u_0$ is not a critical point of $\mathcal{J}$, then
%\begin{equation*}
%\frac{d\mathcal{J}(\Pi(s))}{ds}=\Vert \nabla \mathcal{J}(\Pi(s))\Vert^2_\H> 0.
%\end{equation*}
%\end{proof}

\begin{prop}\label{converg-set2}
Given $u_0\in\H$ which is not a critical point of $\mathcal{J}$, the solution of Eq. (\ref{IVP-grad-flow}) starting from $u_0$ converges to a connected component of the set of critical points of $\mathcal{J}$ as $s\rightarrow +\infty$.
\end{prop}

\begin{proof}[Proof of Proposition \ref{converg-set2}]
If $u_0$ is not a critical point of $\mathcal{J}$, then
\begin{eqnarray}
\frac{d\cJ(\Pi(s))}{ds}&=&(\nabla\cJ(\Pi(s)),\frac{d\Pi(s)}{ds})_\H\nonumber\\
&=&\Vert\nabla\cJ(\Pi(s))\Vert_\H^2>0.\label{eq:increase}
\end{eqnarray}
Since $\O(\rho_0)$ is compact, the real-valued function $J$ defined on $\mathcal{O}(\rho_0)$ is bounded. Therefore, $\cJ=J\circ\End$ is also bounded. (\ref{eq:increase}) implies that $\displaystyle\lim_{s\rightarrow+\infty}\cJ(\Pi(s))$ exists. 

We now show that ${\displaystyle\frac{d\cJ(\Pi(s))}{ds}}$ is uniformly continuous. By Corollary \ref{estimate-var-eq-order1}, Eq. (\ref{eq:estimate}), and Corollary \ref{estimate-var-eq2} respectively, $\End(\cdot)$, $\Pi(\cdot)$, and $G(\cdot)$ are all Lipschitz functions, they are therefore uniformly continuous. As $\nabla J(\cdot)$ is a continuous function defined on the compact set $\O(\rho_0)$, it is also uniformly continuous.  Therefore, ${\displaystyle\frac{dJ(\pi(s))}{ds}}$ is uniformly continuous as composition of uniformly continuous functions. 

Since $\displaystyle\lim_{s\rightarrow+\infty}J(\pi(s))$ exists and ${\displaystyle\frac{dJ(\pi(s))}{ds}}$ is uniformly continuous, Barbalat's Lemma implies that 
$$\lim_{s\rightarrow+\infty}\frac{d\cJ(\Pi(s))}{ds}=0.$$In other words, $\Pi(\cdot)$ converges to a connected component of the set of critical point of $\mathcal{J}$.

\end{proof}

\begin{rem}
The above result only guarantees the convergence of $\Pi(\cdot)$ to a {\it set} of critical points but does not directly imply the existence of $\displaystyle\lim_{s\rightarrow+\infty}\Pi(s)$. We need further information about the set of critical points of $\mathcal{J}$.
\end{rem}

%Before pursuing investigation on the asymptotic behavior of \eqref{IVP-grad-flow1}, we first show two general properties of the gradient flow of $\cJ$ which are important in practice.
%
%\begin{prop}[Regularity]
%
%\end{prop}
%
%\begin{prop}[Finite dimensional reduction]
%
%\end{prop}
\subsection{Characterization of critical points}

\begin{prop}\label{critical-J}
A control $u\in\H$ is a critical point of $\mathcal{J}$ if and only if
\begin{equation}
\nabla J(\End(u))\in\textrm{ Kernel } (d\End^*(u)),
\end{equation}which is equivalent to
\begin{equation}
\nabla J(\End(u))\perp \textrm{ Image } (d\End(u)),
\end{equation}where the orthogonality symbol $\perp$ is taken with respect to the inner product $\langle\cdot,\cdot\rangle_{\End(u)}$. An equivalent condition is that
the switching function $\Phi_{\rho_0,\nabla J(\End(u))}(\cdot)$ is equal to zero almost everywhere on $[0,T]$. 
\end{prop}\smallskip
 
\begin{proof}[Proof of Corollary \ref{critical-J}]

It suffices to note that the kernel of $d\End^*(u)$ is equal to the orthogonal complement of of the image of $d\End(u)$ with respect to the inner product $\langle\cdot,\cdot\rangle_{\End(u)}$. The last condition comes from Corollary \ref{equal}.

\end{proof}\smallskip

Proposition \ref{critical-J} together with Lemma \ref{eq:right-hand-grad} implies the following more explicit characterization.
\begin{coro}\label{coro:critical-J2}
A control $u\in\H$ is a critical point of $\mathcal{J}$ if and only if 
\begin{equation}\label{eq:critical-J3}
\tr~([\rho_0,U(T)^\dagger\theta U(T)]U^\dagger(t)H_1U(t))=0,~\textrm{ for }t\in[0,T],
\end{equation}which is equivalent to
\begin{equation}\label{eq:critical-J4}
\tr~([\rho(T),\theta]U(t-T)^\dagger H_1U(t-T))=0,~\textrm{ for }t\in[0,T].
\end{equation}
\end{coro}\medskip

The following properties are straightforward.
\begin{coro}\label{inclus}
Consider $u\in\H$. If $\End(u)$ is a critical point of $J$, then $u$ is a critical point of $\mathcal{J}$.
\end{coro}

\begin{coro}\label{regular-critical}
If $u\in\H$ is a regular control, then $u$ is a critical point of $\mathcal{J}$ if and only if $\End(u)$ is a critical point of $J$.
\end{coro}\smallskip
%\begin{rem}
%If $[\rho(T),\theta]=0$, i.e. $\rho(T)$ is a critical point of $J$, then Eq. (\ref{eq:critical-J4}) is satisfied, implying that the corresponding $u$ is a critical point of $\mathcal{J}$. This is consistent with Corollary \ref{inclus}. 
%\end{rem}

For later discussion, we distinguish two types of critical points.

\begin{de}\label{kinematic-dynamic}
A control $u\in\H$ is a {\it kinematic critical point} of $\mathcal{J}$ if $\End(u)$ is a critical point of $J$. All other critical points of $\mathcal{J}$ are called {\it dynamic or non-kinematic critical point}.
\end{de}

\begin{rem}\label{rem:sing-reg}
We note that dynamic critical points are necessarily singular in the sense of Definition \ref{de:singular-control} while kinematic critical points can be either regular or singular. In the absence of singular controls in $\H$, all the critical points of $\mathcal{J}$ are kinematic and regular. We also note that the dynamic critical points are necessarily not solutions for {\bf Problem 1}, see Remark \ref{main-rem}.
\end{rem}

\subsection{Analysis in the absence of singular controls}

The standing assumption of this section is the following:
\begin{itemize}
\item[({\bf H2})] {\it all} the controls in $\H$ are regular (Definition \ref{de:regular-control}).
\end{itemize}

Although ({\bf H2}) seems restrictive, it allows us to give a complete analysis of the asymptotic behavior of (\ref{IVP-grad-flow1}) in accordance with existing numerical simulation results. The goal of this section is to prove the following result.
\begin{theo}\label{pt-convergence}
Under ({\bf H1}) and ({\bf H2}), every solution of the gradient flow (\ref{IVP-grad-flow1}) converges to a critical point of $\cJ$ as $s\rightarrow \infty$. Moreover, for almost all initial conditions, the solution of (\ref{IVP-grad-flow1}) converges to a solution of {\bf Problem 1}.
\end{theo}\medskip

We start by giving a more precise characterization of the set of critical points of $\cJ$ under ({\bf H2}). For ease of notation, the kernel of $d\End(u)$ and the image of $d\End^*(u)$ will respectively be denoted by $K_u$ and $\mathcal{I}_u$.
\begin{prop}\label{critical-manifold}
For $i=1,\dots, M$, let $\displaystyle\H_i:=\{u\in\H, ~\End(u)=\rho_i\}$. Under ({\bf H2}), we have
\begin{itemize}
\item[(i)] the set of critical points of $\mathcal{J}$ is the disjoint union of $\H_i$ with $i=1,\dots, M$;
\item[(ii)] $\H_i$'s are submanifolds in $\H$ of co-dimension $N$;
\item[(iii)] The tangent space to $\H_i$ at $u\in\H_i$ denoted by $T_{u}\H_i$ is equal to $K_u$.
\end{itemize}
\end{prop}
\begin{proof}[Proof of Proposition \ref{critical-manifold}]
(i) is a consequence of \ref{regular-critical} and the fact that $J$ has only isolated critical points. For all $u\in\H$, $d\End(u)$ has finite rank, thus its kernel splits. By ({\bf H2}), $\End(\cdot)$ is a submersion from $\H$ to $\mathcal{O}(\rho_0)$ (cf. \cite[Prop. 2.3, p 29]{lang}). Therefore, by the Submersion Theorem, the set $\End^{-1}(\rho_i)$ is a submanifold  in $\H$ of co-dimension $N$, and $T_u\H_i=K_u$ (cf. \cite[Th. 3.5.4, p 175]{Marsden}).  
\end{proof}\smallskip

By computing the second order Taylor expansion of $\mathcal{J}$, the following result holds true.
\begin{lemma}\label{expression-hessian}
For $u\in\H_i$, the Hessian of $\mathcal{J}$ at $u$ is given by 
\begin{equation}\label{eq:hessian}
\A(u):=d\End^*(u)\nabla^2J(\rho_i)d\End(u).
\end{equation}
\end{lemma}

\begin{prop}\label{hessian-kernel}
For $u\in\H_i$, we have
\begin{itemize}
\item[(i)] the kernel of $\A(u)$ is equal to $K_u$; 
\item[(ii)] the number of positive (resp. negative) eigenvalues of $\A(u)$ is equal to the number of positive (resp. negative) eigenvalues of $\nabla^2J(\rho_i)$.
\end{itemize}
\end{prop}
\begin{proof}[Proof of Proposition \ref{hessian-kernel}]
For (i), let $v\in\H$. Since $d\End^*(u)$ is injective, $\A(u)v=0$ implies $\nabla^2J(\rho_i)d\End(u)v=0$. By Lemma \ref{hessian-j}, one gets $v\in K_u$. The converse is clear. For (ii), we first note that by ({\bf H2}) and Lemma \ref{hessian-j}, the image of $\A(u)$ is equal to $\mathcal{I}_u$. Let $g(u)$ be the positive definite symmetric matrix such that $g^2(u)=G(u)$ (see Definition \ref{gramian}). We set $$a(\rho_i):=g(u)\nabla^2J(\rho_i)g(u).$$Since $g(u)=g^T(u)$, by Sylvester's law of inertia, $a(\rho_i)$ and $\nabla^2J(\rho_i)$ have the same numbers of positive and negative eigenvalues. Let $\{\nu_k\}_{k=1,\dots, M}$ be the set of eigenvalues of $a(\rho_i)$ and $\{\mu_k\}_{k=1,\dots,M}$ be the corresponding set of orthonormal eigenvectors. For $k=1,\dots, M$, let 
$$v_k:=d\End^*(u)g(u)^{-1}\mu_k.$$Then, it is clear that the set $\displaystyle\{v_k\}_{k=1,\dots,M}$ forms a basis of $\mathcal{I}_u$. Moreover, we have, for $k=1,\dots,M$,
\begin{eqnarray*}
\A(u)v_k&=&d\End^*(u)\nabla^2J(\rho_i)G(u)g^{-1}(u)\mu_k\\
&=&d\End^*(u)\nabla^2J(\rho_i)g(u)\mu_k\\
&=&\nu_kd\End^*(u)g^{-1}(u)\mu_k\\
&=&\nu_kv_k.
\end{eqnarray*}Therefore, the non-zero part in the spectrum of $\A(u)$ is equal to the spectrum of $a(\rho_i)$. We conclude that $\nabla^2J(\rho_i)$ and the restriction of $\A(u)$ to $\mathcal{I}_u$ have the same signature.

\end{proof}
As a direct consequence of Lemma \ref{hessian-j} and Proposition \ref{hessian-kernel}, we have the following result, which, together with Proposition \ref{hessian-kernel}, will play a crucial role in the convergence analysis of (\ref{IVP-grad-flow1}).
\begin{coro}\label{coro:index}
~
\begin{itemize}
\item[(i)]For $u\in\H_1$, $\A(u)$ restricted to $\mathcal{I}_u$ is positive definite;
\item[(ii)]For $u\in\H_M$, $\A(u)$ restricted to $\mathcal{I}_u$ is negative definite;
\item[(iii)]For $u\in\H_i$ with $i=2,\dots, M-1$, $\A(u)$ restricted to $\mathcal{I}_u$ is not definite.
\end{itemize}
\end{coro}

Based on (i) of Proposition \ref{hessian-kernel}, the following result is a generalization of the classical Morse Lemma (see for example \cite[Ch. 7, Th. 5.1]{lang}). For functions defined on a finite dimensional manifold, a result similar to Proposition \ref{morse-bott} is known as Morse-Bott Lemma. In fact, the following result deals with the case where critical submanifolds are of infinite dimension. For the sake of completeness, a proof will be given in Appendix. 
\begin{prop}\label{morse-bott}
Let $\mathcal{C}$ be a connected component of $\H_i$ and $u_c\in\mathcal{C}$. Then, there exist an open neighborhood $U$ of $u_c$ in $\H$ and a smooth chart $\displaystyle\phi:U\mapsto \H=\mathcal{P}_1\stackrel{\perp}{\oplus}\mathcal{P}_2\stackrel{\perp}{\oplus}\mathcal{K}$ such that
\begin{itemize}
\item[(i)] the dimensions of $\mathcal{P}_1$ and $\mathcal{P}_2$ are equal to $N_1^i$ and $N-N_1^i$ respectively, where $N_1^i$ is the Morse index of $u_c$;
\item[(ii)]$\phi(u_c)=0$, and $$\phi(U\cap\mathcal{C})=\{(v_1,v_2,w)\in\mathcal{P}_1\times\mathcal{P}_2\times\mathcal{K}, ~v_1=v_2=0\};$$
\item[(iii)] $\displaystyle\mathcal{J}\circ\phi^{-1}(v_1,v_2, w)=\mathcal{J}(u_c)-\Vert v_1\Vert_\H^2+\Vert v_2\Vert_\H^2$.
\end{itemize}
\end{prop}

\begin{rem}\label{rem:index}
By Corollary \ref{coro:index}, $N_1^1=0$, $0<N_1^i<N$, for $i=2,\dots, M-1$, and $N_1^M=N$.
\end{rem}

\begin{rem}\label{rect}
The gradient flow defined by $\cJ$ and the one defined by $\tilde{\cJ}:=\cJ\circ\phi^{-1}$ are equivalent in the sense of \cite[Theorem, Ch. 1, Sec. 5.3]{arnold}. The diffeomorphism $\phi$ provides us with a suitable change of coordinates and allows us to simplify the expression of (\ref{IVP-grad-flow1}).
\end{rem}\medskip

Once we have Proposition \ref{morse-bott}, the proof of Theorem \ref{pt-convergence} is a straightforward adaptation of the proof sketch of \cite[Prop. 3.6, Ch. 1, p 20]{helmke-moore}.\smallskip

\begin{proof}[Proof of Theorem \ref{pt-convergence}]
We know from Proposition \ref{converg-set2} that the flow converges to a connected component $\mathcal{C}$ of $\H_i$ for some $i\in\{1,\dots, M\}$. Fix an arbitrary $u_c\in\mathcal{C}$ and consider a neighborhood $U$ of $u_c$ small enough . Without loss of generality, if $\Pi(\cdot)$ is the solution of (\ref{IVP-grad-flow1}), we can assume that $\Pi(s_0)\in U$ for some $s_0>0$ large enough. Using the change of coordinates introduced in Proposition \ref{morse-bott} and taking into account Remark \ref{rect}, the gradient flow of $\cJ$ starting from $\Pi(s_0)$ is equivalent to the gradient flow of $\cJ\circ\phi^{-1}$ in a neighborhood of $\mathcal{C}$,
\begin{equation}\label{new-flow}
\begin{array}{lll}
\dot{v}_1&=&-v_1,\smallskip\\
\dot{v}_2&=&v_2,\smallskip\\
\dot{w}&=&0,
\end{array}
\end{equation}where $(v_1,v_2,w)\in\mathcal{P}_1\times\mathcal{P}_2\times\mathcal{K}$ and $(v_1(0), v_2(0),w(0))=\phi(\Pi(s_0))$. The solution of (\ref{new-flow}) will be denoted by $\displaystyle\tilde{\Pi}(\cdot)$. Two situations can happen: 
\begin{itemize}
\item[(i)] if $\phi(\Pi(s_0))=(v_{1,0},0,w_0)$ for some $v_{1,0}\in\mathcal{P}_1$ and $w_0\in\mathcal{K}$, then $\displaystyle\lim_{s\rightarrow\infty}\tilde{\Pi}(s)=(0,0,w_0)$, which implies that $$\displaystyle\lim_{s\rightarrow\infty}{\Pi}(s)=\phi^{-1}(0,0,w_0)\in\H_i;$$
\item[(ii)] if $\phi(\Pi(s_0))=(v_{1,0},v_{2,0},w_0)$ with $v_{2,0}\neq 0\in\mathcal{P}_2$, then 
\begin{equation}\label{diverg}
 \displaystyle\lim_{s\rightarrow\infty}\tilde{\Pi}(s)=+\infty.
 \end{equation}This case requires that $\mathcal{P}_2$ be a subspace of dimension greater than $1$, i.e., $\mathcal{C}$ is a connected component of $\H_i$ for some $i\in\{1,\dots, M-1\}$, see Remark \ref{rem:index}. However, if this case happens, (\ref{diverg}) implies that $\Pi(\cdot)$ does not converge to $\mathcal{C}$. Therefore, $\mathcal{C}$ is necessarily a connected component of $\H_M$. In this case, the flow (\ref{new-flow}) is reduced to the following
\begin{equation}\label{new-flow2}
\begin{array}{lll}
\dot{v}_1&=&-v_1,\smallskip\\
\dot{w}&=&0,
\end{array}
\end{equation}where $(v_1,w)\in\mathcal{P}_1\times\mathcal{K}$ with $\textrm{dim }\mathcal{P}_1=N$. The asymptotic behavior of (\ref{new-flow2}) implies that $\Pi(s)$ will converge to an element of $\mathcal{C}\subset\H_M$ as $s\rightarrow\infty$.
\end{itemize}\smallskip

We conclude from (i) and (ii) that all the solutions of the gradient flow of $\cJ$ always converge pointwise. It is also clear that for almost all initial conditions, case (ii) happens, i.e., almost all the solutions converge to a maximum of $\cJ$.
\end{proof}

%\begin{rem}
%comment on ({\bf H1}), \cite{DKV}...
%\end{rem}

\subsection{Comments on the role of singular controls}

We give in this section a more explicit characterization of singular controls for \eqref{CS} and then explain some difficulty in the analysis of the asymptotic behavior of \eqref{IVP-grad-flow1} related to their presence in $\H$. We first recall the following result which is a direct application of \cite[Th. 6, p 41]{bonnard} or \cite[Prop. 5.3.4, p 94]{trelat} to \eqref{CS}.

\begin{lemma}\label{char:singular}
Let $u\in\H$ and $\rho(\cdot)$ be the corresponding trajectory. Then, $u$ is a singular control if and only if there exists an absolutely continuous application $q:[0,T]\mapsto T\O(\rho_0)\setminus\{0\}$ such that
\begin{equation}\label{eq:adjoint-vector-bis}
\dot{q}(t)=[H_0+u(t) H_1,q(t)],
\end{equation}and
\begin{equation}\label{eq:constraint}
\langle q(t),[H_1,\rho(t)]\rangle_{\rho(t)}=0,\quad \textrm{ for }t\in[0,T].
\end{equation}
\end{lemma}

Using Definition \ref{Riem-metric} and (\ref{Ad-invariance}), we have the following equivalent characterization of singular controls.

\begin{coro}\label{equiv:singular}
A control $u\in\H$ is singular if and only if there exists $\Omega_0\in\mathfrak{p}\setminus\{0\}$ such that 
\begin{equation}\label{constraint-bis}
\tr(\Omega_0~U^\dagger(t)H_1U(t))=0,\quad \textrm{ for }t\in[0,T],
\end{equation}with $U(\cdot)$ satisfying
\begin{equation}\label{U-bis}
\left\{\begin{array}{lll}
\dot{U}(t)&=&(H_0+u(t) H_1)U(t),\qquad t\in [0,1],\\
U(0)&=&\textrm{Id}.
\end{array}\right.
\end{equation}
\end{coro}\smallskip

\begin{proof}[Proof of Corollary \ref{equiv:singular}]Assume there exists $\Omega_0\in\mathfrak{p}\setminus\{0\}$ such that Eqs. (\ref{constraint-bis}) and (\ref{U-bis}) are satisfied. Let $q(\cdot)$ be the solution of Eq. (\ref{eq:adjoint-vector-bis}) starting from $q(0):=[\rho_0,\Omega_0]$. Then, we have
\begin{eqnarray*}
q(t)&=&U(t)[\rho_0,\Omega_0]U^\dagger(t)\\%=U(t)\rho_0\Omega_0U^\dagger(t)-U(t)\Omega_0\rho_0U^\dagger(t)\\
&=&\rho(t)U(t)\Omega_0U^\dagger(t)-U(t)\Omega_0U^\dagger(t)\rho(t)\\
&=&\ad_{\rho(t)}\Ad_{U(t)}\Omega_0.
\end{eqnarray*}This implies
\begin{eqnarray*}
&&\langle q(t),[H_1,\rho(t)]\rangle_{\rho(t)}\\
&=&-\langle \ad_{\rho(t)}\Ad_{U(t)}\Omega_0, \ad_{\rho(t)}\Ad_{U(t)}U^\dagger(t)H_1U(t)\rangle_{\rho(t)}\\
&=&-\tr(\Omega_0~(U^\dagger(t)H_1U(t))^\mathfrak{p})\\
&=&-\tr(\Omega_0~ÊU^\dagger(t)H_1U(t))=0.
\end{eqnarray*}
By Proposition \ref{char:singular}, $u$ is singular. The converse is immediate.
\end{proof}

\begin{rem}\label{rem:singular}
This result states that a control $u$ is singular if and only if the real and imaginary parts of the matrix elements of the projection of $U^\dagger(t) H_1U(t)$ on $\mathfrak{p}$, where $U(\cdot)$ satisfies \eqref{U-bis}, are $\R-$linearly dependent functions of $t$ over the time interval $[0,T]$. 
\end{rem}

We note that, according to Corollary \ref{inclus}, the elements of $\H_i$, for $i=1,\dots, M$, are still critical points of $\cJ$ called \emph{kinematic} critical points. However, due to possible rank deficiency of the end-point map, the Submersion Theorem may non longer be used and $\H_i$ may not necessarily be submanifolds of $\H$. More importantly, for $u\in\H_i$, although the expression of the Hessian of $\cJ$ at $u$ given by \eqref{eq:hessian} is still valid, the two crucial results given in Proposition \ref{hessian-kernel} may fail. In other words, the non-zero part of the signature of $\mathcal{A}(u)$ may non longer be determined by the signature of $\nabla^2J(\End(u))$. This is the first complication in the asymptotic analysis due to the presence of singular controls.  

The second difficulty is the occurrence of \emph{non-kinematic} critical points of $\cJ$. We know from Remark \ref{main-rem} that these critical points are not global maxima of $\cJ$. However, nothing \emph{a priori} prevents them from being local maxima and then ``attracting" solutions of \eqref{IVP-grad-flow1}. Although this situation has never been observed in numerical simulations, a formal proof is still missing. A complete spectral analysis on the Hessian of $\cJ$ needs to be performed. We note that if $u$ is a non-kinematic critical points of $\cJ$, the Hessian form of $\cJ$ at $u$ is given by
\begin{eqnarray}\label{hessian-J2}
&&\nabla^2\cJ(u)(v,v)\\
&=&(v,\mathcal{A}(u)v)_\H\nonumber+\langle\nabla J(\rho),d^2\End(u)(v,v)\rangle_\rho,
\end{eqnarray}where $v\in\H$ and $\rho:=\End(u)$.

%\subsection{Numerical simulations}

\section{Conclusion}\label{concl}

We presented in this paper a gradient-type dynamical system to solve the problem of maximizing quantum observables ({\bf Problem 1}). Under the regularity assumption on the controls ({\bf H2}), we proved that for almost all initial conditions, Eq. \eqref{IVP-grad-flow1} converges to a solution of {\bf Problem 1}. We also detailed difficulties related to the presence of singular controls, which constitute the starting point for further investigations. From our point of view, one first needs more explicit characterization of singular controls, then deduces information on the ``size" of the set of singular controls $\mathcal{S}$ in the entire control space $L^2([0,T],\R)$. The next step is to investigate the optimality status of a ``generic" elements of $\mathcal{S}$. Finally, let us also emphasize that upon due care to numerical details, simulations for extensive systems always achieved the global maximum.  

\appendix[Proof of Proposition \ref{morse-bott}]
We first note that Morse-Bott Lemma are often stated without proof as a direct consequence of Morse Lemma. A complete proof of this result for functions defined in finite dimensional vector spaces can be found in \cite{Ban-Hur}. We will see in the following that dealing with infinite dimensional critical submanifolds presents no difficulty.\medskip

\begin{proof}[Proof of Proposition \ref{morse-bott}]
Let $\mathcal{C}$ be a connected component of $\H_i$ and $u_c\in\mathcal{C}$. Since $\H_i$ is a submanifold of $\H$ of co-dimension $N$, there exist a neighborhood $U$ of $u_c$ in $\H$ and a smooth chart $\displaystyle\varphi: U\mapsto \H=\mathcal{P}\stackrel{\perp}{\oplus}\mathcal{K}$ such that
\begin{itemize}
\item the dimension of $\mathcal{P}$ is equal to $N$;
\item $\varphi(u_c)=0$, and $\varphi(U\cap\mathcal{C})=\{0\}\times \mathcal{K}$.
\end{itemize}
Fix $w\in\mathcal{K}$ in a neighborhood of $0$. Proposition \ref{hessian-kernel} implies that the Hessian at $0$ of the new functional $\cJ_w$ defined by $$\cJ_w(v):=\cJ\circ\varphi^{-1}(v,w)$$ is non degenerate on $\mathcal{P}$. Note also that the signature of the Hessian of $\cJ_w$ at $0$ is equal to the one of the restriction of $\mathcal{A}(u_c)$ to $\mathcal{I}_{u_c}$. Applying Morse Lemma (cf. \cite[Ch. 7, Th. 5.1]{lang}) to $\cJ_w$, there exists a smooth change of coordinates $\psi_w$, $v\mapsto x:=\psi_w(v)$, such that 
$$\cJ_w(\psi_w^{-1}(x))=(Ax,x),$$where $A$ is a symmetric matrix which has the same signature of the Hessian of $\cJ_w$ at $0$. Note also that $\psi_w$ depends smoothly on $w$.

Let $\varphi(u)=(\varphi_1(u),\varphi_2(u))\in\mathcal{P}\times\mathcal{K}$ and $\phi$ be the new smooth chart for $\mathcal{C}$ in a neighborhood of $u_c$ defined by
$$\phi(u):=(\psi_{\varphi_2(u)}(\varphi_1(u)),\varphi_2(u)).$$ Then, by construction, if $x=\phi(u)$, we have 
\begin{equation}\label{proof-morse-bott}
\cJ\circ\phi^{-1}(x)=(Ax,x).
\end{equation}
Proposition \ref{morse-bott} follows from \eqref{proof-morse-bott}.
\end{proof}

\bibliographystyle{plain}

\bibliography{Note_Regular_Case}

%\begin{IEEEbiography}[{\includegraphics[width=1in,height=1.25in,clip,keepaspectratio]{photolong}}]{Ruixing Long}
%was born in Nanjing, China, in 1984. He received in 2007 the Dipl\^ome d'Ing\'enieur from ENSTA ParisTech, Paris, France, and the MS degree in mathematics from Universit\'e Paris-Sud 11, Orsay, France. He is currently a doctoral student at Ecole Polytechnique, Paris, France. His research interest is the motion planning problem for general nonholonomic systems. 
%\end{IEEEbiography}

\end{document}